\documentclass[11pt]{amsproc}
\usepackage{amssymb}
\usepackage[all]{xy}
\usepackage{mathrsfs}
\usepackage{amsmath,amssymb,amsthm}
\usepackage{extarrows}
\usepackage{mathtools}
\usepackage{tensor}
\newtheorem{theorem}{THEOREM}[section]
\newtheorem{lemma}[theorem]{LEMMA}
\newtheorem{proposition}[theorem]{PROPOSITION}
\newtheorem{definition}[theorem]{\emph{Definition}}
\newtheorem{remark}[theorem]{\emph{Remark}}
\newtheorem{corollary}[theorem]{COROLLARY}
\newtheorem{question}[theorem]{QUESTION}
\def\to{\rightarrow}
\def\ot{\otimes}

\def\c{\mathcal}
\def\b{\mathbf}
\def\r{\mathrm}
\def\bb{\mathbb}

\begin{document}
\title[Contractible Banach Algebras]{A REMARK ON CONTRACTIBLE\\ BANACH ALGEBRAS OF OPERATORS}
\author[M.M. Sadr]{MAYSAM MAYSAMI SADR}
\begin{abstract}
For a Banach algebra $A$, we say that an element $M$ in $A\ot^\gamma A$ is a hyper-commutator if $(a\ot1)M=M(1\ot a)$ for every
$a\in A$. A diagonal for a Banach algebra is a hyper-commutator which its image under diagonal mapping is $1$. It is well-known that
a Banach algebra is contractible iff it has a diagonal. The main aim of this note is to show that for any Banach subalgebra $A\subseteq\c{L}(X)$ of
bounded linear operators on infinite-dimensional Banach space $X$, which contains the ideal of finite-rank operators, the image of
any hyper-commutator of $A$ under the canonical algebra-morphism $\c{L}(X)\ot^\gamma\c{L}(X)\to\c{L}(X\ot^\gamma X)$, vanishes.
\end{abstract}
\maketitle
\section{\textbf{INTRODUCTION}}\label{2211070001}
A Banach algebra $A$ is called contractible (super-amenable) \cite{Runde1,Runde2} if every bounded derivation from $A$ into any
Banach $A$-bimodule, is inner. Contractibility is a strong version of the notion of amenability. The concept of amenability (for Banach algebras)
has been formulated by Johnson in his seminal paper \cite{Johnson1} on Hochschild cohomology of Banach algebras.
For various notions of amenability in Theory of Banach Algebras, see \cite{Mewomo1,Runde1,Runde2}.
It is known that any finite-dimensional contractible Banach algebra is a finite direct sum of full matrix algebras \cite[Theorem 4.1.4]{Runde2}.
Until now, the only known contractible Banach algebras are of this form.
Indeed, it is a longstanding question that whether every contractible Banach algebra is
finite-dimensional \cite[p. 224]{Runde1}. Also, the following special case of this question
has not been answered yet \cite{Gronbaek1},\cite[p. 224]{Runde1}: does for any Banach space $X$, the contractibility of the
Banach algebra $\c{L}(X)$ of all bounded linear operators on $X$, imply that $X$ is finite-dimensional?
For information on these questions see \cite[$\S$4.1 \& p. 196]{Runde2} and \cite{Ozawa1,Johnson2}.
We must remark that the chance that there exist infinite-dimensional
contractible Banach algebras is not very small: For a long time it was a common belief that for infinite-dimensional Banach spaces $X$,
$\c{L}(X)$ can not be amenable. But, in 2009, Argyros and Haydon \cite{Argyros1} found out a specific
infinite-dimensional Banach space $E$ which its dual is $\ell^1=E^*$ and has The Scaler-Plus-Compact Property.
For such a Banach space $E$, as it has been pointed out by Dales, $\c{L}(E)$ is an amenable Banach algebra; see \cite{Runde3}.

In this note, we introduce the notion of hyper-commutator for Banach algebras.
It is well-known that a Banach algebra is contractible iff it is unital and has a diagonal. By definition, a diagonal of a
Banach algebra is a hyper-commutator which its image under the diagonal mapping is $1$. The main aim of this note is to prove the following
property of hyper-commutators: For any infinite-dimensional Banach space $X$, and any Banach subalgebra $A$ of $\c{L}(X)$
which contains the ideal of finite-rank operators, the image of any hyper-commutator of $A$, under the canonical algebra-morphism,
$$A\ot^\gamma A\hookrightarrow\c{L}(X)\ot^\gamma\c{L}(X)\to\c{L}(X\ot^\gamma X),$$ vanishes.
For the proof we use the famous Kadec-Snobar's estimate \cite[Theorem 6.28]{Fabian1} on operator-norms of projections.

Since our results are mainly concerned about contractibility of $\c{L}(X)$, some known results on contractibility
are organized in $\S$\ref{2211070002} for contractible central Banach algebras. (So, there is nothing special new in $\S$\ref{2211070002}.) 
In $\S$\ref{2211070004}, we prove our main result and give some new remarks on contractibility of $\c{L}(X)$.
\section{\textbf{SOME KNOWN RESULTS ON CONTRACTIBILITY}}\label{2211070002}
For preliminaries on contractibility, we refer the reader to Runde's books \cite{Runde1,Runde2}.
(All results in this section are well-known or are simple variations of the results of \cite{Johnson2,Runde1,Runde2,Ozawa1}.)
The topological dual of a Banach space $X$ is denoted by $X^*$. The completed projective tensor product of Banach spaces
$X,Y$ is denoted by $X\ot^\gamma Y$. The projective norm is denoted by $\|\cdot\|_\gamma$.
The Banach space of bounded linear operators from $X$ into $Y$ is denoted by $\c{L}(X,Y)$.
For Banach algebras $A,B$, the Banach space $A\ot^\gamma B$ is a Banach algebra with the multiplication given by
$(a\ot b)(a'\ot b')=aa'\ot bb'$ for $a,a'\in A,b,b'\in B$.
The diagonal mapping $\Delta:A\ot^\gamma A\to A$ for $A$ is the unique bounded linear operator defined by $a\ot b\mapsto ab$.
A diagonal for a unital Banach algebra $A$ is an element $M\in A\ot^\gamma A$ satisfying
\begin{equation*}
\Delta(M)=1,\hspace{10mm}(c\ot 1)M=M(1\ot c),\hspace{5mm}(c\in A).
\end{equation*}
It is well-known that a Banach algebra is contractible iff it is unital and has a diagonal:
Suppose that $A$ is contractible. Let $E$ denote the Banach $A$-bimodule with the underlying Banach space $A$ and, left and right module-operations
$ax:=ax$ and $xa:=0$ for $a\in A,x\in E$. Then $\r{id}:A\to E$ is a derivation and hence inner. Thus $A$ has a right unit.
Similarly, it is proved that $A$ has a left unit and hence $A$ is unital. Now, consider the derivation $D:A\to\ker(\Delta)$  defined by
$a\mapsto (1\ot a)-(a\ot1)$.
(Note that $A\ot^\gamma A$ is canonically a Banach $A$-bimodule with module-operations given by
$c(a\ot b):=(c\ot1)(a\ot b)$ and $(a\ot b)c:=(a\ot b)(1\ot c)$, and $\Delta$ is a bimodule-morphism.)
$D$ must be inner and thus there is $N\in\ker(\Delta)$ with the property $aN-Na=D(a)$; hence $M:=N+(1\ot1)$ is a diagonal for $A$.
Conversely, suppose that $M$,
\begin{equation}\label{p3}
M=\sum_{n=1}^\infty a_n\ot b_n,\hspace{10mm}\sum_{n=1}^\infty\|a_n\|\|b_n\|<\infty,\hspace{5mm}(a_n,b_n\in A)\end{equation}
is a diagonal for $A$. If $D:A\to X$ is a bounded derivation, then it can be checked that for the element $z:=\sum_{n=1}^\infty a_nD(b_n)$ of $X$
we have $D(a)=az-za$. Thus $D$ is inner.
\begin{lemma}\label{2211040916}
Let $A$ be a contractible Banach algebra and let $E,F$ be unital Banach left $A$-modules. Then any diagonal for $A$ gives rise to a
bounded projection $\Phi=\Phi_{E,F}$ from $\c{L}(E,F)$ onto $\tensor[_A]{\c{L}}{}(E,F)$.\end{lemma}
\begin{proof}Let $M$ be a diagonal for $A$ of the form (\ref{p3}). For $T\in\c{L}(E,F)$ let
$$\Phi(T):E\to F,\hspace{10mm}x\mapsto\sum_{n=1}^\infty a_nT(b_nx),\hspace{5mm}(x\in E).$$
Then it can be checked that $\Phi(T)$ is well-defined and belongs to $\tensor[_A]{\c{L}}{}(E,F)$. Also, it is easily verified that
$T\mapsto\Phi(T)$ is a bounded linear projection.\end{proof}
\begin{proposition}\label{2211040917}
Let $A$ be a contractible Banach algebra. Then any diagonal for $A$ gives rise to a canonical bounded linear operator
$\Psi:A\to\c{Z}(A)$ with $\Psi(1)=1$.\end{proposition}
\begin{proof} Consider $A$ as a Banach left $A$-module in the canonical fashion. For any $c\in A$, let $\ell_c:A\to A$ denote the left multiplication
operator by $c$. By the notations of Lemma \ref{2211040916}, $$\Phi_{A,A}(\ell_c):A\to A,\hspace{10mm}x\mapsto\sum_{n=1}^\infty a_ncb_nx$$
is a left module-morphism and hence there is a $\tilde{c}\in A$ such that $\Phi_{A,A}(\ell_c)=r_{\tilde{c}}$ where $r_{\tilde{c}}:A\to A$
denotes the right multiplication operator by $\tilde{c}$. It is clear that $\tilde{c}=\sum_{n=1}^\infty a_ncb_n$ and $\tilde{c}\in\c{Z}(A)$.
We let $\Psi$ to be defined by $c\mapsto\tilde{c}$.\end{proof}
The following result is a variation of \cite[Proposition 5.1]{Johnson2}.
\begin{proposition}\label{2211040918}
Let $A$ be a contractible central Banach algebra. Then any diagonal for $A$ gives rise to a canonical bounded linear functional
$\psi\in A^*$ with $\psi(1)=1$.\end{proposition}
\begin{proof}We have $\c{Z}(A)=\mathbb{C}1$. With the notations of Proposition \ref{2211040917}, $\psi$ is defined by
$$\Psi(c)=\sum_{n=1}^\infty a_ncb_n=\psi(c)1,\hspace{10mm}(c\in A).$$\end{proof}
\begin{theorem}\label{2211040919}
Let $A$ be a contractible central Banach algebra. Then $A$ has a unique maximal (two-sided) ideal $\c{M}_A$.\end{theorem}
\begin{proof}Let $$\c{M}_A:=\text{ closed linear span of }\Big\{c\in A: c \text{ belongs to a proper ideal of } A\Big\}.$$
It is clear that $\c{M}_A$ is an ideal of $A$ which contains every proper ideal of $A$. With $\psi$ as in Proposition \ref{2211040918},
for any $c\in A$ which is contained in a proper ideal $J$ of $A$, we must have $\psi(c)=0$, because otherwise we must have
$1=\psi(c)^{-1}\sum_{n=1}^\infty a_ncb_n\in J$, a contradiction. Thus $\c{M}_A\subseteq\ker(\psi)$
and hence $\c{M}_A$ is a proper ideal of $A$.\end{proof}
A closed linear subspace $F$ of a Banach space $E$ is called topologically complemented if there is a closed linear subspace $F'$ of $E$ such that
$E=F\oplus F'$. In this case $F'$ is called a topological complement for $F$. $F$ is topologically complemented in $E$ iff there is a bounded
linear projection from $E$ onto $F$.
\begin{lemma}\label{2211101430}
Let $A$ ba a contractible Banach algebra and let $E$ be a unital Banach left $A$-module. Suppose that $E$ is compactly generated i.e.
there exists a norm-compact subset $K$ of $E$ such that every $x\in E$ is of the form $x=ay$ for some $a\in A,y\in K$.
Suppose that $E$ has approximation property. Then $E$ is finite-dimensional.\end{lemma}
\begin{proof}
Let $M$ be a diagonal for $A$ of the form (\ref{p3}). We can suppose that $\|b_n\|\to0$ and $\sup_{n\geq1}\|a_n\|<\infty$.
Continuity of the module-operation implies that the set $\cup_{n\geq1}b_nK\subset E$ is contained in a compact subset of $E$.
Let $\Phi_{E,E}$ be defined as in the proof of Lemma \ref{2211040916}. The approximation property for $E$ means that there exists a net
$(S_\lambda)_\lambda$ of finite-rank operators in $\c{L}(E)$ such that $S_\lambda\to\r{id}_E$ uniformly on compact subsets of $E$.
The above assumptions imply that $\Phi_{E,E}(S_\lambda)$ is a net of compact operators on $E$ such that converges
uniformly to $\r{id}_E$ on $K$. Now, since $\Phi_{E,E}(S_\lambda)$'s are module-morphisms and $K$ generates $E$, we have
$\Phi_{E,E}(S_\lambda)\to\r{id}_E$ in operator-norm. Thus, $\r{id}_E$ is a compact operator. Hence, $E$ is finite-dimensional.\end{proof}
Note that any Banach space $X$ considered as a unital Banach left $\c{L}(X)$-module in the canonical fashion,
is generated by any of its nonzero vectors. Also, for any unital Banach algebra $A$ and any closed left ideal $J$ of $A$, the quotient
Banach left $A$-module $A/J$ is generated by the class of $1$ in $A/J$.
\begin{lemma}\label{2211040159}
Let $A$ be a contractible Banach algebra and let $E$ be a unital Banach left $A$-module. Suppose that $F\subset E$ is a closed submodule
which is (as a Banach space) topologically complemented in $E$. Then $F$ has a topological
complement in $E$ which is also a closed submodule.\end{lemma}
\begin{proof}Let $p$ be a bounded linear projection from $E$ onto $F$. By Lemma \ref{2211040916}, $\Phi_{E,F}(p)$ is a module-morphism
from $E$ into $F$. It is easily verified that $\Phi_{E,F}(p)$ is also a projection from $E$ onto $F$. Thus $\ker\Phi_{E,F}(p)$ is the
desired complement for $F$.\end{proof}
If $A,A'$ are contractible Banach algebras with diagonals $M,M'$ of the forms as in (\ref{p3}), then
$\sum_{n,m=1}^\infty a_n\ot a'_m\ot b_n\ot b'_m$ is a diagonal for $A\ot^\gamma A'$. Also, $\sum_{n=1}^\infty b_n\ot a_n$ is a diagonal for
$A^{\r{op}}$, the opposite algebra of $A$. Thus if $A$ is contractible then $A\ot^\gamma A^{\r{op}}$ is contractible.
\begin{lemma}\label{2211040202}
The analogue of Lemma \ref{2211040159} is satisfied for bimodules: Let $A$ be a contractible Banach algebra and $E$
a unital Banach $A$-bimodule. If $F$ is a closed sub-bimodule of $E$ which is topologically complemented, then
it has a complement in $E$ which is also a sub-bimodule.\end{lemma}
\begin{proof}Any unital Banach $A$-bimodule $E$ may be considered as unital Banach left $A\ot^\gamma A^\r{op}$-module with module operation given by
$(a\ot b)x:=axb$ ($a\in A,b\in A^\r{op},x\in E$). In this fashion, any $A$-bimodule-morphism is a left $A\ot^\gamma A^\r{op}$-module-morphism.
The converses of this facts are also satisfied. Now, the desired result follows from Lemma \ref{2211040159}.\end{proof}
\begin{theorem}\label{2211061315}
Let $A$ be a contractible central Banach algebra. Suppose that $J$ is a closed, proper, and nonzero ideal of $A$.
(Note that the existence of $J$ implies that $A$ is infinite-dimensional. Indeed, it follows from \cite[Theorem 4.1.2]{Runde2} that
any contractible central Banach algebra of finite dimension is isomorphic to a full matrix algebra.) The following statements hold.
\begin{enumerate}
\item[(i)] $J$ is not topologically complemented in $A$.
\item[(ii)] $A/J$ has not approximation property.
\item[(iii)] If $J$ is compactly generated as left (resp. right) $A$-module, then $J$ has not approximation property.\end{enumerate}\end{theorem}
\begin{proof}
(i): If $J$ is topologically complemented in $A$, then by Lemma \ref{2211040202} there is a closed, proper, and nonzero ideal $J'$
such that $A=J\oplus J'$. Thus we have $J,J'\subset\c{M}_A$, a contradiction. (Note that (i) may be concluded from centrality
of $A$. Indeed, if $A=J\oplus J'$, then there exist orthogonal nonzero central idempotents $e\in J,e'\in J'$ with $e+e'=1$.)
(ii): If $A/J$ has approximation property, then by
Lemma \ref{2211101430}, $A/J$ is finite-dimensional and hence $J$ is topologically complemented in $A$, a contradiction with (i).
(iii) follows from Lemma \ref{2211101430}, similarly.\end{proof}
The following  corollary follows from the above results.
\begin{corollary}\label{2211071100}
Let $X$ be an infinite-dimensional Banach space. If $\c{L}(X)$ is contractible then, (i) $X$ has not approximation property;
(ii) $\c{L}(X)$ has a unique maximal ideal $\c{M}$; (iii) $\c{M}$ is not topologically complemented;
and (iv) $\c{L}(X)/\c{M}$ has not approximation property.\end{corollary}
A contractible Banach algebra $A$ is called symmetrically contractible if $A$ has a symmetric diagonal; that is, a diagonal $M$
satisfying $\c{F}_A(M)=M$ where $\c{F}_A:A\ot^\gamma A\to A\ot^\gamma A$ denotes flip i.e., the unique bounded linear mapping
defined by $(a\ot b)\mapsto(b\ot a)$. The matrix algebra $\b{M}_n$ is symmetrically contractible. Indeed, it is well-known that $\b{M}_n$
has the unique diagonal $n^{-1}\sum_{i,j=1}^n\delta_{ij}\ot\delta{ji}$ where $\delta_{ij}$'s
denote the standard basis of $\b{M}_n$. Thus any finite-dimensional contractible Banach algebra is symmetrically contractible.

The following results are variations of \cite[Proposition 5.3]{Johnson2}.
\begin{theorem}\label{2211050933}
Let $A$ be a symmetrically contractible Banach algebra. Then any symmetric diagonal of $A$ gives rise to a bounded normalized $\c{Z}(A)$-valued
trace for $A$. If $A$ is central, then $A$ has a normalized trace $\psi\in A^*$.\end{theorem}
\begin{proof}Let $M$ be a symmetric diagonal for $A$ of the form (\ref{p3}). We saw in Proposition \ref{2211040917} that the assignment
$c\mapsto\sum_{n=1}^\infty a_ncb_n$ defines a bounded linear mapping $\Psi:A\to\c{Z}(A)$ with $\Psi(1)=1$. For every $c,c'\in A$ we have
$\sum_{n=1}^\infty b_n\ot a_ncc'=\sum_{n=1}^\infty cb_n\ot a_nc'$ and hence $\sum_{n=1}^\infty a_ncc'\ot b_n=\sum_{n=1}^\infty a_nc'\ot cb_n$. Thus
we have $$\Psi(cc')=\Delta\Big(\sum_{n=1}^\infty a_ncc'\ot b_n\Big)=\Delta\Big(\sum_{n=1}^\infty a_nc'\ot cb_n\Big)=\Psi(c'c).$$
$\psi$ is given as in Proposition \ref{2211040918}.\end{proof}
For the matrix algebra $\b{M}_n$, the unique diagonal of $\b{M}_n$ gives rise to the ordinary trace.
\section{\textbf{A NULL-PROPERTY OF DIAGONALS}}\label{2211070004}
Let $X$ be a Banach space. Consider the unique bounded linear operator $$\Upsilon:\c{L}(X)\ot^\gamma\c{L}(X)\to\c{L}(X\ot^\gamma X),$$ defined by
$$[\Upsilon(T\ot S)](x\ot y)=T(x)\ot S(y),\hspace{10mm}(T,S\in\c{L}(X),x,y\in X).$$
Then $\Upsilon$ is an algebra-morphism between Banach algebras. We denote the image under $\Upsilon$ of any element $N\in\c{L}(X)\ot^\gamma\c{L}(X)$,
by $N^\r{op}$. It follows from properties of projective tensor product, that $\|\Upsilon\|=1$ and hence $\|N^\r{op}\|\leq\|N\|_\gamma$.
Note that, in general, $\Upsilon$ is not one-to-one. (This fact can be concluded from the fact that the canonical mapping from
$X^*\ot^\gamma X^*$ onto the space of nuclear bilinear functionals on $X\times X$ is not necessarily one-to-one \cite[$\S$2.6]{Ryan1}.)
\begin{proposition}\label{2211061350}
Let $\Lambda\in\c{L}(X\ot^\gamma X)$ be such that for every one-rank operator $T\in\c{L}(X)$, $$(T\ot1)^\r{op}\Lambda=\Lambda(1\ot T)^\r{op}.$$
Then there is a unique operator $\Gamma$ in $\c{L}(X)$ such that $\Lambda=(1\ot\Gamma)^\r{op}\c{F}_{X}$.\end{proposition}
\begin{proof} Let $y$ be a nonzero vector in $X$, and let $f\in X^*$ be such that $f(y)=1$. Let $T\in\c{L}(X)$ to be defined by $x\mapsto f(x)y$.
For $x\in X$ we have
\begin{equation}\label{2211061210}
(T\ot1)^\r{op}\Lambda(x\ot y)=\Lambda(x\ot y).\end{equation}
$X$ has the decomposition $<y>\oplus\ker(f)$ where $<y>$ denotes the subspace generated by $y$. There exist
$z\in\ker(f)\ot^\gamma X$ and $w\in X$ such that $$\Lambda(x\ot y)=y\ot w+z.$$ It follows from (\ref{2211061210}) that $\Lambda(x\ot y)=y\ot w$.
Since the mapping $x\mapsto\Lambda(x\ot y)$ is linear and bounded, there is $\Gamma_y\in\c{L}(X)$ such that
$\Lambda(x\ot y)=y\ot\Gamma_y(x)$. Now, suppose that $y,y'$ in $X$ are linearly independent. We have
$$\Lambda(x\ot(y+y'))=y\ot\Gamma_y(x)+y'\ot\Gamma_{y'}(x),$$ $$\Lambda(x\ot(y+y'))=(y+y')\ot\Gamma_{y+y'}(x).$$
Thus $\Gamma_y=\Gamma_{y'}$. Also, it can be checked that for every nonzero scalar $\lambda$ we have $\Gamma_{\lambda y}=\Gamma_y$.
Thus there exists $\Gamma\in\c{L}(X)$ such that $\Lambda(x\ot y)=y\ot\Gamma(x)$ for every $x,y\in X$. The proof is complete.\end{proof}
\begin{corollary}\label{2211061606}
Let $M$ be an element of $\c{L}(X)\ot^\gamma\c{L}(X)$ that satisfies
\begin{equation}\label{2211071000}
(T\ot1)M=M(1\ot T),\hspace{10mm}(T\in\c{L}(X)\hspace{1mm}\text{of rank one}).
\end{equation}
Then there exists $\Gamma\in\c{L}(X)$ such that $M^\r{op}=(1\ot\Gamma)^\r{op}\c{F}_X$. Moreover, if $M$ is symmetric (i.e. $\c{F}_{\c{L}(X)}(M)=M$)
then there exists a scaler $\lambda$ such that $$M^\r{op}=\lambda\c{F}_X.$$\end{corollary}
\begin{proof}The first part follows directly from Proposition \ref{2211061350}. Suppose that $M$ is symmetric. It follows from the identity
$[\c{F}_{\c{L}(X)}(M)]^\r{op}=\c{F}_XM^\r{op}\c{F}_X$, that $$\c{F}_X(1\ot\Gamma)^\r{op}=(1\ot\Gamma)^\r{op}\c{F}_X.$$
Thus for every $x,y\in X$ we have $\Gamma(y)\ot x=y\ot\Gamma(x)$. This means that $\Gamma$ is a scalar multiple of identity.
The proof is complete.\end{proof}
Let $Y,Y',Z$ be finite-dimensional Banach spaces. Similar to the mapping $\Upsilon$ above, we denote by $\Upsilon:N\mapsto N^\r{op}$
the unique bounded linear mapping $$\c{L}(Y,Z)\ot^\gamma\c{L}(Z,Y')\to\c{L}(Y\ot^\gamma Z,Z\ot^\gamma Y'),$$
given by $$(T\ot S)^\r{op}(y\ot z)=(T(y)\ot S(z)).$$ We know that this is a linear isomorphism.
\begin{lemma}\label{2211061420}
With the above assumptions, suppose that $\r{dim}(Y)=\r{dim}(Y')$. Suppose that $T:Y\to Y'$ is a linear isomorphism.
For every finite-dimensional Banach space $Z$, let the linear mapping $\tilde{T}_Z$ be given by
$$\tilde{T}_Z:Y\ot^\gamma Z\to Z\ot^\gamma Y',\hspace{10mm}(y\ot z)\mapsto(z\ot T(y)).$$ There is a numerical positive constant $c$
such that $c$ is independent from $Z$ (independent from norm and dimension of $Z$) and such that:
$$\|\Upsilon^{-1}(\tilde{T}_Z)\|_\gamma\geq c^{-1}\r{dim}(Z).$$\end{lemma}
\begin{proof}Suppose that $y_1,\ldots,y_k$ and $z_1,\ldots,z_m$ are vector basis respectively for $Y$ and $Z$, and let $y'_i=T(y_i)$.
Let the linear operators $$S_{ij}:Y\to Z,\hspace{5mm}S'_{ji}:Z\to Y',\hspace{10mm}(1\leq i\leq k,1\leq j\leq m)$$
be given by $$S_{ij}(y_i)=z_j, S_{ij}(y_q)=0,\hspace{3mm}(q\neq i),\hspace{10mm}S'_{ji}(z_j)=y'_i, S'_{ji}(z_q)=0,\hspace{3mm}(q\neq j).$$
Let $N:=S_{ij}\ot S'_{ji}$. Then $N^\r{op}=\tilde{T}_Z$ and hence $\Upsilon^{-1}(\tilde{T}_Z)=N$.

Let $\nu$ denote the linear functional on $\c{L}(Y,Y')$ that associates to any operator $Y\to Y'$, the normalized trace of its matrix in the bases $y_1,\ldots,y_k$ and $y'_1,\ldots,y'_k$ of $Y$ and $Y'$. Suppose that $c$ denotes the functional-norm of $\nu$. It is clear that $c\neq0$.
Consider the bilinear functional $$\mu:\c{L}(Y,Z)\times\c{L}(Z,Y')\to\bb{C},\hspace{10mm}(P,Q)\mapsto\nu(QP).$$
Then we have $\|\mu\|\leq c$ and hence $\|c^{-1}\mu\|\leq1$. Now, it follows from the properties of projective tensor-norm that
$$\|N\|_\gamma\geq|c^{-1}\mu(N)|=c^{-1}m.$$\end{proof}
\begin{proposition}\label{2211061008}
Let $X$ be an infinite-dimensional Banach space. Let $M\in\c{L}(X)\ot^\gamma\c{L}(X)$ be an element that satisfies (\ref{2211071000}).
Then $M\in\ker(\Upsilon)$. In other notation, $M^\r{op}=0$.\end{proposition}
\begin{proof}Suppose that $\Gamma\in\c{L}(X)$ is as in Corollary \ref{2211061606}. Suppose that $M^\r{op}\neq0$ and hence $\Gamma\neq0$.
Let $y,y'$ be two nonzero vectors in $X$ such that $\Gamma(y)=y'$. Suppose that $Y,Y'$ denote the one-dimensional subspaces of $X$ generated
respectively by $y,y'$, and suppose that $T:Y\to Y'$ is defined by $T(y)=y'$. Let $Z$ be an arbitrary finite-dimensional subspace of $X$.
Suppose that $E_Y:Y\to X$ and $E_Z:Z\to X$ denote the embedding-maps and $P_{Y'}:X\to Y'$ is an arbitrary continuous projection from $X$ onto $Y'$.
By Kadec-Snobar's Theorem \cite[Theorem 6.28]{Fabian1} we know that there exists a continuous projection $P_Z:X\to Z$, from $X$ onto $Z$,
such that $\|P_Z\|<1+\sqrt{\r{dim}(Z)}$.
Let $$N:=(P_Z\ot P_{Y'})M(E_Y\ot E_{Z})\in\c{L}(Y,Z)\ot^\gamma\c{L}(Z,Y').$$
We have $$\|N\|_\gamma\leq\|P_Z\|\|P_{Y'}\|\|M\|_\gamma,\hspace{10mm}N^\r{op}=\tilde{T}_Z,$$
where $\tilde{T}_Z$ is as in Lemma \ref{2211061420}. Now, by Lemma \ref{2211061420} we have
$$\frac{\r{dim}(Z)}{c\|P_{Y'}\|(1+\sqrt{\r{dim}(Z)})}<\|M\|_\gamma.$$
This implies that $\|M\|_\gamma=\infty$, a contradiction. Thus, we have $M^\r{op}=0$.\end{proof}
\begin{definition}
\emph{Let $A$ be a Banach algebra and let $M\in A\ot^\gamma A$. We say that $M$ is a hyper-commutator for $A$
if $$aM=Ma\hspace{10mm}(a\in A).$$ }\end{definition}
By definition, diagonals of contractible Banach algebras are hyper-commutator elements.
Following the discussion of  $\S$\ref{2211070001}, the next question is very natural.
\begin{question}
Does there exist an infinite-dimensional Banach algebra with a nonzero hyper-commutator?\end{question}
The next theorem which is the main result of this note, establishes a null-property of hyper-commutators.
\begin{theorem}\label{2211071001}
Let $X$ be an infinite-dimensional Banach space. Let $A\subseteq\c{L}(X)$ be a Banach subalgebra such that contains
the ideal of finite-rank operators. Then the image of any hyper-commutator of $A$ under the canonical algebra-morphism
$$A\ot^\gamma A\hookrightarrow\c{L}(X)\ot^\gamma\c{L}(X)\to\c{L}(X\ot^\gamma X),$$ vanishes.\end{theorem}
\begin{proof}It follows directly from Proposition \ref{2211061008}.\end{proof}
Note that for any Banach algebra $A$ as in Theorem \ref{2211071001}, we have $$\c{Z}(A)=0\hspace{5mm}\text{or}\hspace{5mm}\c{Z}(A)=1\bb{C}.$$
\begin{remark}
\emph{Suppose that $\c{L}(X)$ is contractible. By Theorem \ref{2211071001}, to prove that $X$ is finite-dimensional, it is enough to prove that
at least one of the diagonals of $\c{L}(X)$ is invertible as a member of the Banach algebra $\c{L}(X)\ot^\gamma\c{L}(X)$.
Note that for the unique diagonal $M$ of $\b{M}_n$ we have $n^2M^2=1\ot1$ in the Banach algebra $\b{M}_n\ot^\gamma\b{M}_n$.}\end{remark}
\begin{remark}
\emph{Suppose that $X$ is an infinite dimensional Banach space for which the canonical mapping $\Upsilon$ is one-to-one.
Then, by Theorem \ref{2211071001}, any Banach subalgebra of $\c{L}(X)$ containing the ideal of finite-rank operators,
is not contractible.}\end{remark}
\begin{remark}
\emph{Let $X$ be an infinite dimensional Banach space. If $A=\c{L}(X)$ has at least two maximal ideals, then by
Corollary \ref{2211071100} we know that $A$ is not contractible. (See \cite{Laustsen1} for some examples of such Banach spaces.)
Suppose that $A$ has only one maximal ideal $J$. To prove that $A$ is not contractible it is enough
to show that the closer $\tilde{J}$ of the ideal $(J\ot A)+(A\ot J)\subset A\ot^\gamma A$ is a maximal ideal of $A\ot^\gamma A$: Indeed,
if $A$ is contractible then $A\ot^\gamma A$ is contractible and since $A\ot^\gamma A$ is central
(this fact can be checked by considering projections onto finite-dimensional subspaces of $X$ similar to the first part of
the proof of Lemma \ref{2211061420}) then it must have a unique maximal ideal. Thus we must have $\ker(\Upsilon)\subseteq\tilde{J}$ and
hence for any diagonal $M$ of $A$, $M$ belongs to $\tilde{J}$.
Therefore, we have $1=\Delta(M)\in\Delta(\tilde{J})\subset J$ that contradicts properness of $J$.}\end{remark}


{\footnotesize
\begin{flushright}
\begin{tabular}{ccccc}
\emph{Maysam Maysami Sadr}\\
\emph{Institute for Advanced Studies in Basic Sciences}\\
\emph{Department of Mathematics}\\
\emph{Zanjan 45137-66731, Iran}\\
\emph{sadr@iasbs.ac.ir}
\end{tabular}
\end{flushright}}
\end{document}